\documentclass[10pt]{amsart}
\usepackage{indentfirst,enumerate,cite,amssymb,amsfonts,amsmath,amsthm,mathrsfs,dsfont}
\usepackage{color}
\usepackage{multicol}
\usepackage[colorlinks=true,linkcolor=blue,citecolor=blue]{hyperref}
\usepackage{enumerate}
\usepackage{geometry}
\numberwithin{equation}{section}

\theoremstyle{plain}
\newtheorem{thm}{Theorem}[section]
\newtheorem{prop}[thm]{Proposition}
\newtheorem{cor}[thm]{Corollary}
\newtheorem{lemma}[thm]{Lemma}
\newtheorem{defi}[thm]{Definition}
\newtheorem{rem}[thm]{Remark}
\newtheorem{ex}[thm]{Example}

\newcommand{\Rep}{\mbox{Rep}}
\newcommand{\HS}{{\mathtt{HS}}}
\newcommand{\Tr}{\mbox{\emph{Tr}}}

\newcommand{\N}{{\mathbb{N}}}

\newcommand{\R}{\mathbb{R}}
\newcommand{\Z}{\mathbb{Z}}
\newcommand{\C}{\mathbb{C}}

\newcommand{\St}{{\mathbb S}^3}

\newcommand{\DG}{\mathcal{D}'(G)}

\newcommand{\jp}[1]{{\left\langle{#1}\right\rangle}}

\newlength{\dhatheight}
\newcommand{\doublehat}[1]{%
	\settoheight{\dhatheight}{\ensuremath{\hat{#1}}}
	\addtolength{\dhatheight}{-0.15ex}
	\widehat{\vphantom{\rule{5pt}{\dhatheight}}%
		\smash{\widehat{#1}}}}

\begin{document}

%
%
%
%
%
%
%
%
%

\title[Global Properties in Komatsu classes]{Global Properties of Vector Fields \\ on Compact Lie Groups in Komatsu classes}


\author[Alexandre Kirilov]{Alexandre Kirilov}
\address{
	Universidade Federal do Paran\'{a}, 
	Departamento de Matem\'{a}tica,
	C.P.19096, CEP 81531-990, Curitiba, Brazil
}
\email{akirilov@ufpr.br}


\author[Wagner de Moraes]{Wagner A. A. de Moraes}
\address{
	Universidade Federal do Paran\'{a},
	Programa de P\'os-Gradua\c c\~ao de Matem\'{a}tica,
	C.P.19096, CEP 81531-990, Curitiba, Brazil
}
\email{wagneramat@gmail.com}


\author[Michael Ruzhansky]{Michael Ruzhansky}
\address{Ghent University, 
	Department of Mathematics: Analysis, Logic and Discrete Mathematics, 
	Ghent, Belgium 
	and 
	Queen Mary University of London, 
	School of Mathematical Sciences, 
	London, United Kingdom
}
\email{Michael.Ruzhansky@ugent.be}




\begin{abstract}
	In this paper we characterize completely the global hypoellipticity and global solvability in the sense of Komatsu (of Roumieu and Beurling types) of constant-coefficients vector fields on compact Lie groups. We also analyze the influence of perturbations by lower order terms in the preservation of these properties. 
\end{abstract}

\maketitle
\tableofcontents

\section{Introduction}

The characterization of global properties for vector fields on closed smooth manifolds has mobilized the efforts of many mathematicians in recent decades. The first step in this direction was taken by Greenfield and Wallach who obtained necessary and sufficient conditions for the global hypoellipticity of constant vector fields on tori. Moreover, these authors conjectured that the existence of a globally hypoelliptic vector field $X$ on a closed smooth manifold $M$ would be equivalent to saying that $M$ is diffeomorphic to a torus and that $X$ is $C^{\infty}-$conjugated to a Diophantine vector field, see  \cite{GW72,GW73b,F08}.

The global hypoellipticity and global solvability on tori, for different classes of vector fields with variable coefficients, has come a long way since then. Among the most inspiring references to this study we mention \cite{BCP04,BDGK15,BDG17,Ber94,GPY92,HouZug17,Hou79,Hou82,Pet11,Petr06,PetrZan08}. 

More recently, we started the study of such properties on compact Lie Groups and we obtained necessary and sufficient conditions to have global hypoellipticity and global solvability for vector fields and their perturbations by low-order terms on a product of two Lie groups, see \cite{KMR19,KMR19b}. 

One of our main results states that a constant coefficient vector field  $L=X_1+a X_2$, defined on $\mathcal{D}'(G_1\times G_2)$ is globally hypoelliptic if and only if its matrix symbol is singular for only finitely many entries and its inverse has at most a polynomial growth with respect to the eigenvalues of the Laplacian on the group.

Since many results about global properties on the torus have a version in Gevrey classes, see for example \cite{AKM19,AlbJor14,Ara18,BDG18,BarFerPet17,GPY93}, we decided to study the ultra-differentiable version of our results and, in view of \cite{DR14,DR16,DR18}, we were naturally led to investigating the Komatsu versions of such properties, obtaining more general results.

To this end we have organized this paper as follows: In Section 2 we present some classical results about Fourier analysis on compact Lie groups and fix the notation that will be used throughout the text. We also present briefly the Komatsu classes of Roumieu and Beurling type. In Section 3 we state the results that give us necessary and sufficient conditions for the global hypoellipticity and the global solvability in Komatsu sense of constant-coefficient vector fields defined on compact Lie group. In Section 4 and 5 we prove the results stated in Section 3. Finally, in Section 6 we study perturbations of vector fields by low order terms, both by constants and by functions.

\raggedbottom
	\section{Preliminaries}
	
	In this section we introduce most of the notations and preliminary results necessary for the development of this study. A very careful presentation of these concepts and the demonstration of all the results presented here can be found in the references \cite{FR16}  and \cite{livropseudo}.
	
	Let $G$ be a compact Lie group and let $\Rep(G)$  be the set of continuous irreducible unitary representations of $G$. Since $G$ is compact, every continuous irreducible unitary representation $\phi$ is finite dimensional and it can be viewed as a matrix-valued function $\phi: G \to \C^{d_\phi\times d_\phi}$, where $d_\phi = \dim \phi$. We say that $\phi \sim \psi$ if there exists an unitary matrix $A\in C^{d_\phi \times d_\phi}$ such that $A\phi(x) =\psi(x)A$, for all $x\in G$. We will denote by $\widehat{G}$ the quotient of $\Rep(G)$ by this equivalence relation.
	
	For $f \in L^1(G)$ the group Fourier transform of $f$ at $\phi \in \Rep(G)$ is
	\begin{equation*}
	\widehat{f}(\phi)=\int_G f(x) \phi(x)^* \, dx,
	\end{equation*}
	where $dx$ is the normalized Haar measure on $G$.
	By the Peter-Weyl theorem, we have that 
	\begin{equation*}\label{ortho}
	\mathcal{B} := \left\{\sqrt{\dim \phi} \, \phi_{ij} \,; \ \phi=(\phi_{ij})_{i,j=1}^{d_\phi}, [\phi] \in \widehat{G} \right\}
	\end{equation*}
	is an orthonormal basis for $L^2(G)$, where we pick only one matrix unitary representation in each class of equivalence, and we may write
	\begin{equation*}
	f(x)=\sum_{[\phi]\in \widehat{G}}d_\phi \emph{\Tr}(\phi(x)\widehat{f}(\phi)).
	\end{equation*}
	Moreover, the Plancherel formula holds:
	\begin{equation*}
	\label{plancherel} \|f\|_{L^{2}(G)}=\left(\sum_{[\phi] \in \widehat{G}}  d_\phi \ 
	\|\widehat{f}(\phi)\|_{\HS}^{2}\right)^{\tfrac{1}{2}}=:
	\|\widehat{f}\|_{\ell^{2}(\widehat{G})},
	\end{equation*}
	where 
	\begin{equation*} \|\widehat{f}(\phi)\|_{\HS}^{2}=\emph{\Tr}(\widehat{f}(\phi)\widehat{f}(\phi)^{*})=\sum_{i,j=1}^{d_\phi}  \bigr|\widehat{f}(\phi)_{ij}\bigr|^2.
	\end{equation*}
	
		The group Fourier transform of $u\in \DG$ at a matrix unitary representation $\phi$ is the matrix $\widehat{u}(\phi) \in \C^{d_\phi \times d_\phi}$, whose components are given by
	$$
	\widehat{u}(\phi)_{ij} = \jp{u,\overline{\phi_{ji}}}.
	$$
	
	Let $\mathcal{L}_G$ be the Laplace-Beltrami operator of $G$. For each $[\phi] \in \widehat{G}$, its matrix elements are eigenfunctions of $\mathcal{L}_G$ correspondent to the same eigenvalue that we will denote by $-\nu_{[\phi]}$, where $\nu_{[\phi]} \geq 0$. Thus
	$$
	-\mathcal{L}_G \phi_{ij}(x) = \nu_{[\phi]}\phi_{ij}(x), \quad \textrm{for all } 1 \leq i,j \leq d_\phi,
	$$
	and we will denote by
	$$
	\jp \phi := \left(1+\nu_{[\phi]}\right)^{1/2}
	$$
	the eigenvalues of $(I-\mathcal{L}_G)^{1/2}.$ We have the following estimate for the dimension of $\phi$: there exists $C>0$ such that for all $[\xi] \in \widehat{G}$ we have
	\begin{equation*}\label{dimension}
	d_\phi \leq C \jp{\phi}^{\frac{\dim G}{2}},
	\end{equation*}
	(see Proposition 10.3.19 of \cite{livropseudo}).
	For $x\in G$, $X\in \mathfrak{g}$ and $f\in C^\infty(G)$, define 
	$$
	L_Xf(x):=\frac{d}{dt} f(x\exp(tX))\bigg|_{t=0}.
	$$
	
	The operator $L_X$ is left-invariant, that is, $\pi_L(y)L_X = L_X\pi_L(y)$, for all $y \in G$. When there is no possibility of ambiguous meaning, we will write only $Xf$ instead of $L_Xf$.

	Let  $P: C^{\infty}(G) \to C^{\infty}(G)$ be a continuous linear operator. The  symbol of the operator $P$ in $x\in G$ and $\phi \in \mbox{{Rep}}(G)$, $\phi=(\phi_{ij})_{i,j=1}^{d_\phi}$ is
	$$
	\sigma_P(x,\phi) := \phi(x)^*(P\phi)(x) \in \C^{d_\phi \times d_\phi},
	$$
	where $(P\phi)(x)_{ij}:= (P\phi_{ij})(x)$, for all $1\leq i,j \leq d_\phi$, and we have
	$$
	Pf(x) = \sum_{[\phi] \in \widehat{G}} \dim (\phi) \mbox{Tr} \left(\phi(x)\sigma_P(x,\phi)\widehat{f}(\phi)\right)
	$$
	for every $f \in C^\infty(G)$ and $x\in G$.
	
	 When $P: C^\infty(G) \to \C^\infty(G)$ is a continuous linear left-invariant operator, that is $P\pi_L(y)=\pi_L(y)P$, for all $y\in G$, we have that $\sigma_P$ is independent of $x\in G$ and
	$$
	\widehat{Pf}(\phi) = \sigma_P(\phi)\widehat{f}(\phi),
	$$
	for all $f \in C^\infty(G)$ and $[\phi] \in \widehat{G}$.
	
	Let $X\in \mathfrak{g}$. It is easy to see that the operator $iX$ is symmetric on $L^2(G)$. Hence, for all $[\phi] \in \widehat{G}$ we can choose a representative $\phi$ such that $\sigma_{iX}(\phi)$ is a diagonal matrix, with entries $\lambda_m(\phi) \in \R$, $1 \leq m \leq d_\phi$. By the linearity of the symbol, we obtain
	$$
	\sigma_X(\phi)_{mn} = i\lambda_m(\phi) \delta_{mn}, \quad \lambda_j(\phi) \in \R.
	$$
	Notice that $\{\lambda_m(\phi)\}_{m=1}^{d_\phi}$ are the eigenvalues of $\sigma_{iX}(\phi)$ and then are independent of the choice of the representative, since the symbol of equivalent representations are similar matrices. Moreover, since $-(\mathcal{L}_G - X^2)$ is a positive operator and commutes with $X^2$, we have
\begin{equation}\label{symbol}
	|\lambda_m(\phi)| \leq \|X\|\jp{\phi},
\end{equation}
	for all $[\phi] \in \widehat{G}$ and $1 \leq m \leq d_\phi$, where $\|\cdot\|$ is the norm induced by the Killing form (see \cite{livropseudo}).

	Let $G_1$ and $G_2$ be compact Lie groups and set $G=G_1\times G_2$. Given $f \in L^1(G)$ and  $\xi \in {\Rep}(G_1)$, the partial Fourier coefficient of $f$ with respect to the first variable  is defined by 
	$$
	\widehat{f}(\xi, x_2) = \int_{G_1} f(x_1,x_2)\, \xi(x_1)^* \, dx_1 \in \C^{d_\xi \times d_\xi}, \quad x_2 \in G_2,
	$$
	with components
	$$
	\widehat{f}(\xi, x_2)_{mn} = \int_{G_1} f(x_1,x_2)\, \overline{\xi(x_1)_{nm}} \, dx_1, \quad 1 \leq m,n\leq d_\xi.
	$$
	Analogously we define the partial Fourier coefficient of $f$ with respect to the second variable. Notice that, by definition, $\widehat{f}(\xi,\: \cdot \:)_{mn} \in C^\infty(G_2)$ and $\widehat{f}(\: \cdot \:, \eta)_{rs} \in C^\infty(G_1)$. 
	
	Let $u \in \DG$, $\xi \in {\Rep}(G_1)$ and $1\leq m,n \leq d_\xi$. The $mn$-component of  the partial Fourier coefficient of $u$ with respect to the first variable is the linear functional defined by
	$$
	\begin{array}{rccl}
	\widehat{u}(\xi, \: \cdot\: )_{mn}: & C^\infty(G_2) & \longrightarrow & \C \\
	& \psi & \longmapsto & \jp{\widehat{u}(\xi, \: \cdot \:)_{mn},\psi} := \jp{u,\overline{\xi_{nm}}\times\psi}_G.
	\end{array}
	$$
	In a similar way, for $\eta \in {\Rep}(G_2)$ and $1 \leq r,s\leq d_\eta$, we define the $rs$-component of the partial Fourier coefficient of $u$ with respect to the second variable.
	It is easy to see that $ \widehat{u}(\xi, \: \cdot\: )_{mn} \in \mathcal{D}'(G_2)$ and $\widehat{u}( \: \cdot\:,\eta )_{rs} \in \mathcal{D}'(G_1)$. 
	
	Notice that
	\begin{equation*}
	\doublehat{\,u\,}\!(\xi,\eta)_{mn_{rs}} = \doublehat{\,u\,}\!(\xi,\eta)_{rs_{mn}} =  \widehat{u}(\xi \otimes \eta)_{ij},
	\end{equation*}
	with $i = d_\eta(m-1)+ r$ and $j  =  d_\eta(n-1) + s$, whenever $u \in C^\infty(G)$ or $u \in \DG$. The details about the partial Fourier series theory can be found in \cite{KMR19}.
	
	In this manuscript we intend to extend Theorems 3.3 and 3.5 of \cite{KMR19b} to Komatsu classes of both Roumieu and Beurling types. First we introduce these space and most of the notation that will be used in the sequel. All definitions are taken from \cite{DR16}, \cite{Kom73} and \cite{Rou60}.
	
	Let $\{ M_k\}_{k \in\N_0}$ be a sequence of positive numbers such that there exist $H>0$ and $A\geq1$ satisfying
	\begin{description}
		\item[(M.0)] $M_0=1$
		\item[(M.1)] (stability)	$M_{k+1} \leq AH^kM_k, \quad k=0,1,2,\dots\,.$ 
		\item[(M.2)] $M_{2k} \leq AH^{2k}M_k^2, \quad k=0,1,2,\dots\,.$
		\item[(M.3)] $\exists \ell, C>0$ such that $k! \leq C\ell^k M_k,$ for all $k \in \N_0$. 
	\end{description}
	
	We will assume also the logarithmic convexity:
	\begin{description}
		\item[(LC)]
		$
		M_k^2 \leq M_{k-1}M_{k+1}, \quad k=1,2,3,\dots .
		$
	\end{description}
	
	Given any sequence $\{M_k\}$ that satisfies (M.0)--(M.3), there exists an alternative sequence that satisfies the logarithmic convexity and defines the same classes that we will study. So assuming (LC) does not restrict the generality compared to (M.0)--(M.3).
	
	From (M.0) and (LC) we have $M_k \leq M_{k+1}$, for all $k\in \N$, that is, $\{M_k\}$ is a non-decreasing sequence. Moreover, for $k \leq n$ we have
	$$
	M_k \cdot M_{n-k} \leq M_n.
	$$
	
	The condition (M.2) is equivalent to
	$
	M_k \leq AH^k \min\limits_{0\leq q \leq k} M_qM_{k-q}, 
	$ (see \cite{PV84}, Lemma 5.3).
	
	Given a sequence $\{M_k \}$ we define the associated function as
	\begin{equation}\label{associated}
	M(r):= \sup_{k\in \N_0} \log \frac{r^k}{M_k}, \quad r>0,
	\end{equation}
	and $M(0):=0$. Notice that $M$ is a non-decreasing function and by its definition,  for every $r>0$ we have
	\begin{equation*}\label{inf}
	\exp\{M(r)\} = \sup_{k\in \N_0} \frac{r^k}{M_k}
	\quad \mbox{and} \quad
	\exp\{-M(r)\} = \inf_{k\in \N_0} \frac{M_k}{r^k}.
	\end{equation*}

Following from these properties we have that for a compact Lie group $G$, for every $p,q, \delta>0$ there exists $C>0$ such that we have
\begin{equation}\label{propM1}
		\jp{\phi}^p \exp\{-\delta M(q\jp{\phi})\} \leq C,
\end{equation}
		for all $[\phi] \in \widehat{G}$. Moreover, for every $q>0$ we have
		\begin{equation}\label{komine}
		\exp\left\{-\tfrac{1}{2} M\left(q\jp{\phi}\right) \right\} \leq \sqrt{A}\exp\{-M\left(q_2 \jp{\phi}\right) \},
		\end{equation}
		for all $[\phi] \in \widehat{G}$, where $q_2= \dfrac{q}{H}$ (see \cite{DR14}).

	\begin{defi}
		The Komatsu class of Roumieu type $\Gamma_{\{M_k\}}(G)$ is the space of all complex-valued $C^{\infty}$ functions $f$ on $G$ such that there exist $h>0$ and $C>0$ satisfying
		$$
		\| \partial^\alpha f\|_{L^2(G)} \leq Ch^{|\alpha|}M_{|\alpha|},  \quad \forall \alpha \in \N_0^n.
		$$
	\end{defi}

	In the definition above, we could take the $L^\infty$-norm and obtain the same space. The elements of $\Gamma_{\{M_k\}}(G)$ are often called ultradifferentiable functions and can be characterized by their Fourier coefficients as follows:
\begin{align}
\label{roum}f \in \Gamma_{\{M_k\}}(G) \iff \exists N>0, \ \exists C>0; \quad
|\doublehat{\,f\,}\!(\xi,\eta)_{mn_{rs}}  | \leq C \exp\{-M(N(\jp{\xi}+\jp{\eta})) \},\\ \forall [\xi] \in \widehat{G_1}, \ [\eta] \in \widehat{G_2}, \ 1\leq m,n\leq d_\xi, \ 1 \leq r,s\leq d_\eta. \nonumber
\end{align}
Similarly, the ultradistribution of Roumieu type can be characterized in the following way:
\begin{align}
\label{distrou}u \in \Gamma'_{\{M_k\}}(G) \iff \forall N>0,\  \exists C_N>0; \quad
|\doublehat{\,u\,}\!(\xi,\eta)_{mn_{rs}}  | \leq C_N \exp\{M(N(\jp{\xi}+\jp{\eta})) \},\\ \forall [\xi] \in \widehat{G_1}, \ [\eta] \in \widehat{G_2}, \ 1\leq m,n\leq d_\xi, \ 1 \leq r,s\leq d_\eta. \nonumber
\end{align}	 
	Next, to define Komatsu classes of Beurling type, we have to change the condition (M.3) by the following one:
	 \begin{description}
	 	\item[(M.3')] $\forall \ell>0$, $\exists C_\ell$ such that $k! \leq C_\ell \ell^k M_k,$ for all $k \in \N_0$.
	 \end{description}
	 Notice that the condition (M.3') implies the condition (M.3).
	 \begin{defi}
	 	The  Komatsu class of Beurling type $\Gamma_{(M_k)}(G)$ is the space of $C^{\infty}$ functions $f$ on $G$ such that for every $h>0$ there exists $C_h>0$ such that we have
	 	$$
	 	\| \partial^\alpha f\|_{L^2(G)} \leq C_hh^{|\alpha|}M_{|\alpha|}, \quad   \forall \alpha \in \N_0^n.
	 	$$
	 \end{defi}
	 
	 Notice that $\Gamma_{(M_k)}(G) \subset \Gamma_{\{M_k\}}(G)$.
The elements of $\Gamma_{(M_k)}(G)$ can be characterized by their Fourier coefficients as follows:
\begin{align}
\label{caracbeurling}f \in \Gamma_{(M_k)}(G) \iff \forall N>0, \ \exists C_N>0; \quad
|\doublehat{\,f\,}\!(\xi,\eta)_{mn_{rs}} | \leq C_N \exp\{-M(N(\jp{\xi}+\jp{\eta})) \},\\ \forall [\xi] \in \widehat{G_1}, \ [\eta] \in \widehat{G_2}, \ 1\leq m,n\leq d_\xi, \ 1 \leq r,s\leq d_\eta. \nonumber
\end{align}
Similarly, the ultradistribution of Beurling type can be characterized in the following way:
\begin{align}
\label{distbeur}u \in \Gamma'_{(M_k)}(G) \iff \exists N>0,\  \exists C>0; \quad
|\doublehat{\,u\,}\!(\xi,\eta)_{mn_{rs}}  | \leq C \exp\{M(N(\jp{\xi}+\jp{\eta})) \},\\ \forall [\xi] \in \widehat{G_1}, \ [\eta] \in \widehat{G_2}, \ 1\leq m,n\leq d_\xi, \ 1 \leq r,s\leq d_\eta. \nonumber
\end{align}

\section{Results}
Let $G_1$ and $G_2$ be compact Lie groups, $G:=G_1\times G_2$, and consider the linear operator $L:C^\infty(G)\to C^\infty(G)$  defined by
\begin{equation*}
L:=X_1+a X_2,
\end{equation*}
where $X_1 \in \mathfrak{g}_1$, $X_2 \in \mathfrak{g}_2$ and $a \in \C$. Thus, for each $u\in C^\infty(G)$ we have
\begin{align}
Lu(x_1,x_2) &:= X_1u(x_1,x_2) + a X_2u(x_1,x_2) \nonumber \\
&:= \frac{d}{dt} u(x_1\exp(tX_1), x_2) \bigg|_{t=0} + a \frac{d}{ds} u(x_1, x_2\exp(sX_2))\bigg|_{s=0}. \nonumber
\end{align}

For each $[\xi] \in \widehat{G_1}$, we can choose a representative $\xi \in \mbox{Rep}({G_1})$ such that
\begin{equation*}\label{symbol1}
\sigma_{X_1}(\xi)_{mn} =i\lambda_m(\xi) \delta_{mn}, \quad 1 \leq m,n \leq d_\xi,
\end{equation*}
where  $\lambda_m(\xi)\in\R$  for all $[\xi] \in\widehat{G_1}$ and $1 \leq m \leq d_\xi$.
Similarly, for each $[\eta] \in \widehat{G_2}$, we can choose a representative $\eta \in \mbox{Rep}({G_2})$ such that
\begin{equation*}\label{symbol2}
\sigma_{X_2}(\eta)_{rs} =i\mu_r(\eta) \delta_{rs}, \quad 1 \leq r,s \leq d_\eta,
\end{equation*}
where  $\mu_r(\eta)\in\R$  for all $[\eta] \in\widehat{G_2}$ and $1 \leq r \leq d_\eta$.

Suppose that $u \in C^\infty(G)$. Thus, taking the partial Fourier coefficient with respect to the first variable we obtain
\begin{equation*}
\widehat{Lu}(\xi,x_2) = \sigma_{X_1}(\xi)\widehat{u}(\xi,x_2)+a X_2 \widehat{u}(\xi,x_2), \quad x_2 \in G_2. \nonumber 
\end{equation*}
Hence, for each $x_2 \in G_2$, $\widehat{Lu}(\xi,x_2) \in \C^{d_\xi\times d_\xi}$  and 
\begin{equation*}
\widehat{Lu}(\xi,x_2)_{mn} = i\lambda_m(\xi)\widehat{u}(\xi,x_2)_{mn}+a X_2 \widehat{u}(\xi,x_2)_{mn}, \quad 1\leq m,n \leq d_\xi.
\end{equation*}

Now, taking the Fourier coefficient of $\widehat{Lu}(\xi, \cdot)_{mn}$ with respect to the second variable, we obtain
\begin{equation*}
\doublehat{\,Lu\,}\!(\xi,\eta)_{mn}=i\lambda_m(\xi)\doublehat{\, u \,}\!(\xi,\eta)_{mn}+ a \sigma_{X_2}(\eta)\doublehat{\, u \,}\!(\xi,\eta)_{mn}. \nonumber
\end{equation*}
Thus, $\doublehat{\,Lu\,}\!(\xi,\eta)_{mn} \in \C^{d_\eta\times d_\eta}$ and 
\begin{equation}\label{fourierLu}
\doublehat{\,Lu\,}\!(\xi,\eta)_{mn_{rs}}=i(\lambda_m(\xi)+a \mu_r(\eta))\doublehat{\, u\,}\!(\xi,\eta)_{mn_{rs}},\quad 1\leq r,s \leq d_\eta.
\end{equation}

Observe that if $\lambda_m(\xi)+a \mu_r(\eta) \neq 0$, then
\begin{equation}\label{fourieru}
\doublehat{\, u \,}\!(\xi,\eta)_{mn_{rs}} = \dfrac{1}{i(\lambda_m(\xi)+a\mu_{r}(\eta))} \doublehat{\,Lu\,}\!(\xi,\eta)_{mn_{rs}}.
\end{equation}

If  $\lambda_m(\xi)+a \mu_r(\eta) = 0$, then \begin{equation}\label{image0}
\doublehat{\,Lu\,}\!(\xi,\eta)_{mn_{rs}}=0.
\end{equation}

Finally, it is clear that what was done above also holds for $u \in \DG$.

If we restrict the operator $L=X_1 + a X_2$ to the Komatsu class of Roumieu type $\Gamma_{\{M_k\}}(G)$ we obtain an endomorphism, that is, $L: \Gamma_{\{M_k\}}(G) \to \Gamma_{\{M_k\}}(G)$. In this way, we can extend the operator $L$ to $u\in\Gamma'_{\{M_k\}}(G)$ as
$$
\jp{Lu,\varphi}:= -\jp{u,L\varphi}, \quad \forall \varphi \in \Gamma_{\{M_k\}}(G).
$$

	\begin{defi}
	Let $G$ be a compact Lie group. We say that an operator $P:\Gamma'_{\{M_k\}}(G) \to \Gamma'_{\{M_k\}}(G)$ is globally ${\Gamma_{\{M_k\}}}$-hypoelliptic if the conditions $u \in \Gamma'_{\{M_k \}}(G)$ and 
	$Pu \in \Gamma_{\{M_k\}}(G)$ imply that $u \in \Gamma_{\{M_k\}}(G)$.
\end{defi}

Our first result presents necessary and sufficient conditions for the global $\Gamma_{\{M_k\}}$--hypoellipticity of the operator $L$.

\begin{thm}\label{thm4}
	The operator $L=X_1+a X_2$ is globally $\Gamma_{\{M_k\}}-$hypoelliptic if and only the following conditions are satisfied:
	\begin{enumerate}[1.]
		\item The set
		$$
		\mathcal{N}=\{([\xi], [\eta]) \in  \widehat{G_1} \times \widehat{G_2}; \ \lambda_m(\xi)+a \mu_r(\eta) = 0, \ \mbox{for some } 1 \leq m \leq d_\xi, 1 \leq r \leq d_\eta  \}
		$$
		is finite.
		\item $\forall N>0, \exists C_N>0$ such that
		\begin{equation}\label{hypothesis4}
		|\lambda_m(\xi)+a\mu_r(\eta)|\geq C_N \exp \{-M(N (\langle \xi \rangle + \langle \eta \rangle))\}, 
		\end{equation}
		for all  $[\xi] \in \widehat{G_1}, \ [\eta] \in \widehat{G_2},  \ 1 \leq m \leq d_\xi, \ 1 \leq r \leq d_\eta$ whenever $\lambda_m(\xi)+a\mu_r(\eta) \neq 0$.
	\end{enumerate}
\end{thm}

As a consequence of this theorem, it will be proved that any globally hypoelliptic constant coefficient vector field is globally $\Gamma_{\{M_k\}}$--hypoelliptic.

Now, to define global solvability for the operator $L$ in the sense of Komatsu classes, observe that given an ultradifferentiable function (or ultradistribution) $f$ defined on $G$, if $u\in\DG$ is a solution of $Lu=f$, we obtain from \eqref{image0} that
$$
\lambda_m(\xi)+a \mu_r(\eta)=0 \Longrightarrow \doublehat{\,f\,}\!(\xi,\eta)_{mn_{rs}}=0.
$$

Therefore, let us consider the following set 
$$\mathcal{K}:= \{f\in \Gamma'_{\{M_k\}}(G) ;  \doublehat{\,f\,}\!(\xi,\eta)_{mn_{rs}}\!=0 \textrm{ whenever } \lambda_m(\xi)+a \mu_r(\eta)=0, \, \mbox{for all}\, 1 \leq m,n \leq d_\xi, \, 1\leq r,s \leq d_\eta\}.
$$ 
Clearly there are no $u \in \Gamma'_{\{M_k\}}(G)$ satisfying $Lu=f$ when $f \notin \mathcal{K}$.

\begin{defi}
	We say that the operator $L$ is globally $\Gamma'_{\{M_k\}}$--solvable if ${L(\Gamma'_{\{M_k\}}(G))=\mathcal{K}}$.
\end{defi}

Notice that $L(\Gamma'_{\{M_k\}}(G)) \subseteq \mathcal{K}$ and the next result gives us the condition to obtain the other inclusion.

\begin{thm}\label{solvrou}
	The operator $L=X_1+a X_2$ is  globally $\Gamma'_{\{M_k\}}(G)$-solvable if and only if \eqref{hypothesis4} holds, that is, for all $N>0$ there exists $C_N>0$ such that we have
	\begin{equation}\tag{\ref{hypothesis4}}
	|\lambda_m(\xi)+a\mu_r(\eta)|\geq C_N \exp \{-M(N (\langle \xi \rangle + \langle \eta \rangle))\}, 
	\end{equation}
	for all  $[\xi] \in \widehat{G_1}, \ [\eta] \in \widehat{G_2},  1 \leq m \leq d_\xi,  1 \leq r \leq d_\eta$ whenever $\lambda_m(\xi)+a\mu_r(\eta) \neq 0$.
	
	Moreover, if $L$ is globally $\Gamma'_{\{M_k\}}(G)$-solvable, then for any $f \in \mathcal{K}\cap \Gamma_{\{M_k\}}(G)$ there exists $u \in \Gamma_{\{M_k\}}(G)$ such that $Lu=f$. 
\end{thm}

From Theorems \ref{thm4} and \ref{solvrou}, we have that the global $\Gamma_{\{M_k\}}$-hypoellipticity of $L$ implies its global $\Gamma'_{\{M_k\}}$-solvability. Moreover, from \eqref{propM1} and the characterization of the global $\mathcal{D}'$--solvability of $L$ given in \cite{KMR19b}, if $L$ is globally $\mathcal{D}'$-solvable, then $L$ is globally $\Gamma'_{\{M_k\}}$-solvable.

For the sequence $M_k = (k!)^s$, with $s \geq 1$, the Komatsu class of Roumieu type is the well-known Gevrey class of order $s$ and when $s=1$ we obtain the class of analytic functions on $G$. For Gevrey classes, we have $M(r) \simeq r^{1/s}$ and \eqref{hypothesis4} becomes
$$
		|\lambda_m(\xi)+a\mu_r(\eta)|\geq C_N e^{-(N (\langle \xi \rangle + \langle \eta \rangle))^{1/s}}. 
$$

Analogously to the Roumieu type case, restricting the operator ${L=X_1 + a X_2}$ to the Komatsu class of Beurling type $\Gamma_{(M_k)}(G)$ we obtain an endomorphism, that is, ${L: \Gamma_{(M_k)}(G) \to \Gamma_{(M_k)}(G)}$. In this way, we can extend the operator $L$ to $u\in\Gamma'_{(M_k)}(G)$ as
$$
\jp{Lu,\varphi}:= -\jp{u,L\varphi}, \quad \forall \varphi \in \Gamma_{(M_k)}(G).
$$

	 \begin{defi}
	Let $G$ be a compact Lie group. We say that an operator ${P:\Gamma'_{(M_k)}(G) \to \Gamma'_{(M_k)}(G)}$ is globally ${\Gamma_{(M_k)}}$-hypoelliptic if the conditions $u \in \Gamma'_{(M_k)}(G)$ and $Pu \in \Gamma_{(M_k)}(G)$ imply that $u \in \Gamma_{(M_k)}(G)$.
\end{defi}

\begin{thm}\label{thm5}
	The operator $L=X_1+a X_2$ is globally $\Gamma_{(M_k)}$-hypoelliptic if and only if the following conditions are satisfied:
	\begin{enumerate}[1.]
		\item The set
		$$
		\mathcal{N}=\{([\xi], [\eta]) \in  \widehat{G_1} \times \widehat{G_2}; \ \lambda_m(\xi)+a \mu_r(\eta) = 0, \ \mbox{for some } 1 \leq m \leq d_\xi, 1 \leq r \leq d_\eta  \}
		$$
		is finite.
		\item $\exists C>0, N>0$ such that
		\begin{equation}\label{hypothesis5}
		|\lambda_m(\xi)+a\mu_r(\eta)|\geq C \exp \{-M(N (\langle \xi \rangle + \langle \eta \rangle))\}, 
		\end{equation}
		for all  $[\xi] \in \widehat{G_1}, \ [\eta] \in \widehat{G_2},  1 \leq m \leq d_\xi,  1 \leq r \leq d_\eta$ whenever $\lambda_m(\xi)+a\mu_r(\eta) \neq 0$.
	\end{enumerate}
\end{thm}

Notice that the conditions for the global $\Gamma_{\{M_k\}}-$hypoellipticity of the Theorem \ref{thm4} imply the conditions for the  global $\Gamma_{(M_k)}-$hypo\-ellipti\-ci\-ty of the Theorem \ref{thm5}. In this way, if $L$ is globally $\Gamma_{\{M_k\}}-$hypo\-ellip\-tic, then $L$ is globally $\Gamma_{(M_k)}-$hypo\-ellip\-tic.

Now, to study the global solvability of $L$ in Komatsu classes of Beurling type, we define 
$$
\mathcal{K}:= \{f\in \Gamma'_{(M_k)}(G) ; \ \doublehat{\,f\,}\!(\xi,\eta)_{mn_{rs}}=0 \textrm{ whenever } \lambda_m(\xi)+a \mu_r(\eta)=0\}.
$$
So, if $f \notin \mathcal{K}$ then there are no $u \in \Gamma'_{(M_k)}(G)$ satisfying $Lu=f$.

\begin{defi}
	We say the operator $L$ is globally $\Gamma'_{(M_k)}$-solvable if $L(\Gamma'_{(M_k)}(G))=\mathcal{K}$.
\end{defi}

Again, we always have $L(\Gamma'_{(M_k)}(G)) \subseteq \mathcal{K}$. The next result gives us the condition for the other inclusion.
\begin{thm}\label{solvabeur}
	The operator $L=X_1 + a X_2$ is globally $\Gamma'_{(M_k)}$-solvable if and only if \eqref{hypothesis5} holds, that is, there exist $C,\, N>0$ such that
	\begin{equation}\tag{\ref{hypothesis5}}
	|\lambda_m(\xi)+a\mu_r(\eta)|\geq C \exp \{-M(N (\langle \xi \rangle + \langle \eta \rangle))\} 
	\end{equation}
	holds for all  $[\xi] \in \widehat{G_1}, \ [\eta] \in \widehat{G_2},  1 \leq m \leq d_\xi,  1 \leq r \leq d_\eta,$ whenever $\lambda_m(\xi)+a\mu_r(\eta) \neq 0$.
	
	Moreover, if $L$ is globally $\Gamma'_{(M_k)}$--solvable, then for an admissible ultradifferentiable function $f \in  \Gamma_{(M_k)}(G)$,  there exists  $u \in \Gamma_{(M_k)}(G)$ such that $Lu=f$. 
\end{thm}

	Notice that if $L$ is globally $\Gamma_{(M_k)}$-hypoelliptic, then $L$ is globally  $\Gamma'_{(M_k)}$--sol\-va\-ble. Moreover, 	if $L$ is globally $\Gamma'_{\{M_k\}}$-solvable, then $L$ is globally $\Gamma'_{(M_k)}$-solvable.

\begin{rem}   
In the article \cite{KMR19b}, to construct examples of globally hypoelliptic vector fields on compact Lie groups, the condition that the set $\mathcal{N} $ has finitely many elements led us to restrict our study to the torus. For this reason, we believe that  the Greenfield-Wallach conjecture is related to this condition. 
Since this same condition on $\mathcal{N}$ appears in the study of the global hypoellipticity in Komatsu classes, we believe that this conjecture extends to these classes. 

Moreover, following the same ideas as in Section 4 of \cite{KMR19b}, it is possible to consider weaker notions of global hypoellipticity in Komatsu classes (hypoellipticity modulo kernel and $\mathcal{W}-$hypoellipticity).

In order to construct examples of global hypoelliptic operators in Komatsu classes defined out of tori, in Section \ref{low-order-pert} we study low-order perturbations of vector fields.
\end{rem}

\section{Proofs in Komatsu classes of Roumieu type}
The main objective of this section is to demonstrate Theorems \ref{thm4} and \ref{solvrou} stated in the previous section and to discuss some of their consequences. Thus, in Propositions \ref{nechrou1} and \ref{nechrou2} we present the  necessary conditions for the operator $L$ to be globally $\Gamma_{\{M_k\}}$--hypoelliptic,  and in Proposition \ref{voltahyporou} we show that these conditions are also sufficient for its global $\Gamma_{\{M_k\}}$--hypoellipticity. We conclude this section with Propositions \ref{voltasolvrou} and \ref{necsolrou}, that present the necessary and sufficient conditions for the operator $L$ to be globally $\Gamma'_{\{M_k\}}$--  
solvable.

	In \cite{KMR19b} it was proved that if the set 	\begin{equation*}\label{setN}
	\mathcal{N}=\{([\xi], [\eta]) \in  \widehat{G_1} \times \widehat{G_2}; \ \lambda_m(\xi)+a \mu_r(\eta) = 0, \ \mbox{for some } 1 \leq m \leq d_\xi, 1 \leq r \leq d_\eta  \}
	\end{equation*}
	has infinitely many elements, then there exists $u \in \DG \backslash C^\infty(G)$ such that $Lu=0$. As such distribution $u$ also belongs to   $\Gamma'_{\{M_k\}(G)} \setminus \Gamma_{\{M_k\}}(G)$, we obtain the following necessary condition for the global $\Gamma_{\{M_k\}}$--hypoellipticity of $L$:
\begin{prop}\label{nechrou1}
	If $L$ is globally $\Gamma_{\{M_k\}}$--hypoelliptic then $\mathcal{N}$ is a finite set.
\end{prop}

In the next proposition we give the second necessary condition for the global $\Gamma_{\{M_k\}}$--hypoellipticity of $L$.

\begin{prop}\label{nechrou2}
	If $L$ is globally $\Gamma_{\{M_k\}}$--hypoelliptic then for any $N>0$ there exits $C_N>0$ such that
	\begin{equation}\tag{\ref{hypothesis4}}
	|\lambda_m(\xi)+a\mu_r(\eta)|\geq C_N \exp \{-M(N (\langle \xi \rangle + \langle \eta \rangle))\}
	\end{equation}
	holds for all  $[\xi] \in \widehat{G_1}, \ [\eta] \in \widehat{G_2},  \ 1 \leq m \leq d_\xi, \ 1 \leq r \leq d_\eta$ whenever $\lambda_m(\xi)+a\mu_r(\eta) \neq 0$.
\end{prop}
\begin{proof}
	Let us prove it by contradiction. Suppose that there exists $N >0$ such that for all $K\in\N$ there exist $[\xi_K] \in \widehat{G_1}$ and $[\eta_K] \in\widehat{G_2}$ satisfying
	\begin{equation}\label{2notsa2}
	0<|\lambda_m(\xi_K)+a\mu_r(\eta_K)| < \textstyle\frac{1}{K}\exp\{-M(N(\langle \xi_K \rangle + \langle \eta_K\rangle))\},
	\end{equation}
	for some $1 \leq m \leq d_{\xi_K}$ and $1 \leq r \leq d_{\eta_K}$. We can suppose that $([\xi_K],[\eta_K]) \notin \mathcal{N}$ and that  $\jp{\xi_j}+\jp{\eta_j} \leq \jp{\xi_\ell} + \jp{\eta_\ell}$ when $j\leq \ell$.
	
	Let $K \in\N$ and $\tilde{m}$ and $\tilde{r}$ be such that \eqref{2notsa2} holds. Define 
	$$\doublehat{\,f\,}\!(\xi_K,\eta_K)_{mn_{rs}} := \left\{
	\begin{array}{ll}
	(\lambda_m(\xi_K)+a \mu_r(\eta_K))(\langle \xi_K \rangle + \langle \eta_K\rangle), & \mbox{if } (m,n)=(\tilde{m},1) \mbox{ and } (r,s)=(\tilde{r},1)\\
	0, & \mbox{otherwise. }
	\end{array}
	\right.
	$$
	
	Let  $C>0$ be obtained from \eqref{propM1} satisfying
	\begin{equation*}\label{constant}
	(\langle \xi_K \rangle + \langle \eta_K\rangle)\exp\left\{-\textstyle\frac{1}{2}M(N(\langle \xi_K \rangle + \langle \eta_K\rangle))\right\} < C,
	\end{equation*}
	for all $K \in \N$. Hence
	\begin{align}
	|\doublehat{\,f\,}\!(\xi_K,\eta_K)_{\tilde{m}1_{\tilde{r}1}}|	&= |\lambda_{\tilde{m}}(\xi_K)+a \mu_{\tilde{r}} (\eta_K)|(\langle \xi_K \rangle + \langle \eta_K\rangle) \nonumber \\
	&\leq  \frac{1}{K}\exp\left\{-M(N(\langle \xi_K \rangle + \langle \eta_K\rangle))\right\}(\jp{\xi_K}+\jp{\eta_K}) \nonumber \\
	&\leq  C \exp\{-M(N(\langle \xi_K \rangle + \langle \eta_K\rangle))\}\exp\{\tfrac{1}{2}M(N(\langle \xi_K \rangle + \langle \eta_K\rangle))\} \nonumber \\
	&\leq C \exp\{-M(\tilde{N}(\jp{\xi_K}+\jp{\eta_K})) \} \nonumber ,
	\end{align}
	where the last inequality comes from \eqref{komine} with $\tilde{N}=\frac{N}{H}$.
	Thus $f \in \Gamma_{\{M_k\}}(G).$
	
	By \eqref{fourieru} and \eqref{image0}, if $Lu=f$ for some $u \in \Gamma'_{\{M_k \}}(G)$, we have
	$$
	\doublehat{\, u \,}\!(\xi_K,\eta_K)_{mn_{rs}} = \left\{
	\begin{array}{ll}
	-i(\langle \xi_K \rangle + \langle \eta_K\rangle), & \mbox{if } mn=\tilde{m}1, rs=\tilde{r}1 \\
	0, & \mbox{otherwise. }
	\end{array}
	\right.
	$$
	In particular,
	\begin{equation}\label{contra2}
	\left|\doublehat{\, u \,}\!(\xi_K,\eta_K)_{\tilde{m}1_{\tilde{r}1}} \right|= \langle \xi_K \rangle + \langle \eta_K\rangle,
	\end{equation}
	for all $K \in \N$. Thus
	$$
	|\doublehat{\, u \,}\!(\xi,\eta)_{mn_{rs}}| \leq \jp{\xi}+\jp{\eta},
	$$ 
	for all $[\xi] \in \widehat{G_1}$, $[\eta] \in \widehat{G_2}$, $1\leq m,n \leq d_\xi$ and $1\leq r,s \leq d_\eta$. Therefore $u \in\DG$ and then $u \in \Gamma_{\{M_k\}}'(G)$. By \eqref{contra2} $u \notin C^{\infty}(G)$. Consequently $ u \notin \Gamma_{\{M_k\}}(G)$, which contradicts the fact that $L$ is globally $\Gamma_{\{M_k\}}$-hypoelliptic.
\end{proof}

To conclude the proof of Theorem \ref{thm4}, let us show that the necessary conditions for the global $\Gamma_{\{M_k\}}$--hypoellipticity for $L$ presented above are also sufficient conditions.

\begin{prop}\label{voltahyporou}
	Assume that the set $\mathcal{N}$ is finite and for any $N>0$ there exits $C_N>0$ such that
	\begin{equation}\tag{\ref{hypothesis4}}
	|\lambda_m(\xi)+a\mu_r(\eta)|\geq C_N \exp \{-M(N (\langle \xi \rangle + \langle \eta \rangle))\}, 
	\end{equation}
	for all  $[\xi] \in \widehat{G_1}, \ [\eta] \in \widehat{G_2},  \ 1 \leq m \leq d_\xi, \ 1 \leq r \leq d_\eta$ whenever $\lambda_m(\xi)+a\mu_r(\eta) \neq 0$, then $L$ is globally $\Gamma_{\{M_k\}}$--hypoelliptic.
\end{prop}
\begin{proof}
Suppose $Lu=f \in \Gamma_{\{M_k\}}(G)$ for some $u \in \Gamma_{\{M_k\}}(G)$. Since $\mathcal{N}$ is finite, it is enough to study the behavior of $\widehat{u}(\xi,\eta)_{mn_{rs}}$ outside of $\mathcal{N}$. If $([\xi],[\eta]) \notin \mathcal{N}$, by \eqref{fourieru} we have
$$
\doublehat{\, u \,}\!(\xi,\eta)_{mn_{rs}} = -i( \lambda_m (\xi) + a \mu_r(\eta))^{-1} \doublehat{\,f\,}\!(\xi,\eta)_{mn_{rs}},
$$
for all $1 \leq m \leq d_\xi$ and $1 \leq r \leq d_\eta$. Thus
\begin{align}
|\doublehat{\, u \,}\!(\xi,\eta)_{mn_{rs}}|	& =  |{\lambda_m(\xi)+ a \mu_r(\eta)}|^{-1} |\doublehat{\,f\,}\!(\xi,\eta)_{mn_{rs}}|\nonumber \\
&\leq C_N\exp\{M(N(\langle \xi\rangle +\langle \eta \rangle))\} |\doublehat{\,f\,}\!(\xi,\eta)_{mn_{rs}}| \nonumber
\end{align}

Since $f \in\Gamma_{\{M_k\}}(G)$, by \eqref{roum} there exist constants $C,N'>0$ such that
$$
|\doublehat{\,f\,}\!(\xi,\eta)_{mn_{rs}}| \leq C \exp\{-M(N'(\langle \xi \rangle + \langle \eta \rangle)) \}.
$$
Hence
$$
|\doublehat{\, u \,}\!(\xi,\eta)_{mn_{rs}}|	\leq C_N\exp\{M(N(\langle \xi\rangle +\langle \eta \rangle))\}\exp\{-M(N'(\langle \xi \rangle + \langle \eta \rangle)) \}.
$$

From \eqref{komine}, for $N=\frac{N'}{H}$, we obtain
$$
\exp\{-M(N'(\langle \xi\rangle +\langle \eta \rangle))\} \leq A\exp\{-2M(N(\langle \xi\rangle +\langle \eta \rangle))\}.
$$
Thus
\begin{equation*}
|\doublehat{\, u \,}\!(\xi,\eta)_{mn_{rs}}|	\leq C\exp\{-M(N(\langle \xi \rangle + \langle \eta \rangle)) \}. \nonumber 
\end{equation*}
Therefore $u \in \Gamma_{\{M_k\}}(G)$. 
\end{proof}

\begin{rem}\label{GHGRH}
	If $L$ is globally hypoelliptic, then $L$ is globally $\Gamma_{\{M_k\}}$-hypoelliptic.
Indeed, by Theorem 3.3 of \cite{KMR19b}, if $L$ is globally hypoelliptic, the set $\mathcal{N}$ is finite and there exist $C,N'>0$ such that
\begin{equation*}
|\lambda_m(\xi)+a\mu_r(\eta)|\geq C (\langle \xi\rangle +\langle \eta \rangle )^{-N'}, 
\end{equation*}
for all  $[\xi] \in \widehat{G_1}, \ [\eta] \in \widehat{G_2}, \ 1 \leq m \leq d_\xi, \  1 \leq r \leq d_\eta,$ whenever $\lambda_m(\xi)+a\mu_r(\eta) \neq 0$.

By  \eqref{propM1}, for every $N>0$, there exits $C_N>0$ such that
$$
(\jp{\xi}+\jp{\eta})^{N'}\exp \{-M(N(\jp{\xi}+\jp{\eta})) \} \leq C_N.
$$
Thus,
\begin{equation*}
|\lambda_m(\xi)+a\mu_r(\eta)|\geq C_N \exp \{-M(N(\jp{\xi}+\jp{\eta})), 
\end{equation*}
for all  $[\xi] \in \widehat{G_1}, \ [\eta] \in \widehat{G_2}, \ 1 \leq m \leq d_\xi, \  1 \leq r \leq d_\eta,$ whenever $\lambda_m(\xi)+a\mu_r(\eta) \neq 0$.
By Theorem \ref{thm4} the operator $L$ is globally $\Gamma_{\{M_k\}}$-hypoelliptic.
\end{rem}

Now let us prove the Theorem \ref{solvrou} about the global $\Gamma'_{\{M_k\}}$--solvability of $L$.
\begin{prop}\label{voltasolvrou}
	Assume that for any $N>0$ there exits $C_N>0$ such that
	\begin{equation}\tag{\ref{hypothesis4}}
	|\lambda_m(\xi)+a\mu_r(\eta)|\geq C_N \exp \{-M(N (\langle \xi \rangle + \langle \eta \rangle))\}, 
	\end{equation}
	for all  $[\xi] \in \widehat{G_1}, \ [\eta] \in \widehat{G_2},  \ 1 \leq m \leq d_\xi, \ 1 \leq r \leq d_\eta$ whenever $\lambda_m(\xi)+a\mu_r(\eta) \neq 0$, then $L$ is globally $\Gamma_{\{M_k\}}$--solvable.
\end{prop}
\begin{proof}
For each $f \in \mathcal{K}$ define
	\begin{equation}\label{solutionrou}
	\doublehat{\, u \,}\!(\xi,\eta)_{mn_{rs}} = \left\{
	\begin{array}{ll}
	0, & \mbox{if } \lambda_m(\xi)+ a \mu_r(\eta)=0, \\
	-i({\lambda_m(\xi)+ a \mu_r(\eta)})^{-1} \doublehat{\,f\,}\!(\xi,\eta)_{mn_{rs}}, & \mbox{otherwise. }
	\end{array}
	\right.
	\end{equation}
	Let us show that $\{ \doublehat{\, u \,}\!(\xi,\eta)_{mn_{rs}} \}$ is the sequence of Fourier coefficients of an element $u \in \Gamma'_{\{M_k\}}(G).$  If $\lambda_m(\xi)+ a \mu_r(\eta)\neq 0$, by \eqref{hypothesis4} we have 
	\begin{align}
	|\doublehat{\, u \,}\!(\xi,\eta)_{mn_{rs}}|
	&=|{\lambda_m(\xi)+ a \mu_r(\eta)}|^{-1} |\doublehat{\,f\,}\!(\xi,\eta)_{mn_{rs}}| \nonumber \\
	& {\leq}  C_N\exp{\{M(N(\langle \xi \rangle + \langle \eta \rangle))\}}  |\doublehat{\,f\,}\!(\xi,\eta)_{mn_{rs}}|. \nonumber 
	\end{align} 
	Using the fact that $f \in \Gamma'_{\{M_k\}}(G)$, we conclude that for all $N>0$ and $N'>0$, there exist $C_{NN'}>0$ such that
	$$
	|\doublehat{\, u \,}\!(\xi,\eta)_{mn_{rs}}| \leq C_{NN'}\exp{\{M(N(\langle \xi \rangle + \langle \eta \rangle))\}}\exp{\{M(N'(\langle \xi \rangle + \langle \eta \rangle))\}},
	$$
	for all $[\xi] \in \widehat{G_1}$, $[\eta] \in \widehat{G_2}$, $1 \leq m,n \leq d_\xi$ and $1 \leq r,s \leq d_\eta$.
	
	Let $D>0$. Choose $N=N'=\frac{D}{H}$. Using \eqref{komine} we obtain
	\begin{align*} 
	|\doublehat{\, u \,}\!(\xi,\eta)_{mn_{rs}}| &\leq C_D \exp\left\{{2M\left(\textstyle\frac{D}{H}(\langle \xi \rangle + \langle \eta \rangle)\right)}\right\}\\
	&\leq C_D \exp{\{M(D(\langle \xi \rangle + \langle \eta \rangle))\}}.
	\end{align*} 
	
	Therefore $u \in \Gamma'_{\{M_k\}}(G)$ and $Lu=f$.
	\end{proof}

\begin{prop}\label{necsolrou}
	If $L$ is globally $\Gamma'_{\{M_k\}}$--solvable, then for any $N>0$ there exits $C_N>0$ such that
	\begin{equation}\tag{\ref{hypothesis4}}
	|\lambda_m(\xi)+a\mu_r(\eta)|\geq C_N \exp \{-M(N (\langle \xi \rangle + \langle \eta \rangle))\}, 
	\end{equation}
	for all  $[\xi] \in \widehat{G_1}, \ [\eta] \in \widehat{G_2},  \ 1 \leq m \leq d_\xi, \ 1 \leq r \leq d_\eta$ whenever $\lambda_m(\xi)+a\mu_r(\eta) \neq 0$. 	Moreover, for any admissible ultradifferentiable function $f \in  \Gamma_{(M_k)}(G)$,  there exists  $u \in \Gamma_{(M_k)}(G)$ such that $Lu=f$. 
\end{prop}
\begin{proof} Suppose that is not true, then there exists $N>0$ such that for all $K\in \N $ there exist $[\xi_K] \in\widehat{G_1}$ and $[\eta_K] \in \widehat{G_2}$ satisfying
	\begin{equation}\label{condition5}
	0 < | \lambda_{\tilde{m}}(\xi_K)+ a \mu_{\tilde{r}}(\eta_K)| \leq \frac{1}{K}\exp\{-M(N(\langle \xi_K \rangle + \langle \eta_K \rangle))\},
	\end{equation}
	for some $1\leq \tilde{m} \leq d_{\xi_K}$ and $1 \leq \tilde{r} \leq d_{\xi_K}$. We can assume that $\jp{\xi_j}+\jp{\eta_j} \leq \jp{\xi_\ell} + \jp{\eta_\ell}$ when $j\leq \ell$. Consider $f\in \mathcal{K}$ defined by
	$$
	\doublehat{\,f\,}\!(\xi,\eta)_{mn_{rs}} = \left\{
	\begin{array}{ll}
	1, &  \mbox{if $([\xi],[\eta])=([\xi_j],[\eta_j])$ for some $j\in \N$ and \eqref{condition5} is satisfied}, \\
	0, & \mbox{otherwise. }
	\end{array}
	\right.
	$$
	Suppose that there exits $u \in \Gamma'_{\{M_k\}}(G)$ such that $Lu=f$. In this way, its Fourier coefficients must satisfy
	$$
	i(\lambda_m(\xi)+ a \mu_r(\eta))\doublehat{\, u \,}\!(\xi,\eta)_{mn_{rs}}=\doublehat{\,f\,}\!(\xi,\eta)_{mn_{rs}}.
	$$
	So
	\begin{align}
	|\doublehat{\, u \,}\!(\xi_K,\eta_K)_{\tilde{m}1_{\tilde{r}1}}| &= |{\lambda_{\tilde{m}}(\xi_K)+ a \mu_{\tilde{r}}(\eta_K)}|^{-1}| |\doublehat{\,f\,}\!(\xi_K,\eta_K)_{\tilde{m}1_{\tilde{r}1}}| \nonumber\\
	& \geq  K\exp\{ M(N(\langle \xi_K \rangle + \langle \eta_K \rangle))\}, \nonumber 
	\end{align}
	which, by Proposition \ref{distrou}, implies that $u \notin \Gamma'_{\{M_k\}}(G)$. Therefore $L$ is not globally solvable.

	Let us now prove the last part of the theorem. Let $f \in \mathcal{K}\cap  \Gamma_{\{M_k\}}(G)$ and define $u$ as in \eqref{solutionrou}. Since $L$ is globally $\Gamma'_{\{M_k\}}$--solvable, it holds \eqref{hypothesis4} and then 
	$$
	|\doublehat{\, u \,}\!(\xi,\eta)_{mn_{rs}}|
	\leq C_N\exp{\{M(N(\langle \xi \rangle + \langle \eta \rangle))\}}  |\doublehat{\,f\,}\!(\xi,\eta)_{mn_{rs}}|,
	$$
	for all  $[\xi] \in \widehat{G_1}, \ [\eta] \in \widehat{G_2},  1 \leq m \leq d_\xi,  1 \leq r \leq d_\eta$. By \eqref{roum}, there exist $C,N'>0$ such that
	$$
	|\doublehat{\,f\,}(\xi_,\eta)_{mn_{rs}}| \leq C\exp\{-M(N'(\jp{\xi}+\jp{\eta})) \}.
	$$
	By  \eqref{komine}, we have for $N=\frac{N'}{H}$ that 
	\begin{align*}
	|\doublehat{\, u \,}\!(\xi,\eta)_{mn_{rs}}| &\leq C\exp\left\{M\left(\tfrac{N'}{H}(\langle \xi \rangle + \langle \eta \rangle)\right)\right\} \exp\{-M(N'(\jp{\xi}+\jp{\eta})) \}\\ 
	&\leq 
	C\exp\left\{M\left(\tfrac{N'}{H}(\langle \xi \rangle + \langle \eta \rangle)\right)\right\} C\exp\left\{-2M\left(\tfrac{N'}{H}(\langle \xi \rangle + \langle \eta \rangle)\right)\right\} \\
	&\leq
	C\exp\left\{-M\left(\tfrac{N'}{H}(\langle \xi \rangle + \langle \eta \rangle)\right)\right\}
	\end{align*}
	for all  $[\xi] \in \widehat{G_1}, \ [\eta] \in \widehat{G_2},  1 \leq m \leq d_\xi,  1 \leq r \leq d_\eta$. Therefore $Lu=f$ and $u\in \Gamma_{\{M_k\}}(G)$.
\end{proof}

\begin{rem}
	Proceeding as in Remark \ref{GHGRH}, we have that if $L$ is globally $\mathcal{D}'$--solvable then $L$ is globally $\Gamma_{\{M_k\}}'$--solvable, (see Theorem 3.5 of \cite{KMR19b}).
\end{rem}

\section{Proofs in Komatsu classes of Beurling type}

Following a script similar to that proposed in the previous section, we dedicate this section to the proof of Theorems \ref{thm5} and \ref{solvabeur}.
Thus, in Propositions \ref{nechbeu1}, \ref{nechbeu2} and \ref{sufhbeu} we present the necessary and sufficient conditions for the operator $L$ to be globally $\Gamma_{(M_k)}$--
hypoelliptic. Next, in Propositions \ref{voltasolvbeur} and \ref{idasolvbeu} we show that the conditions presented for the operator $L$ to be globally $\Gamma'_{(M_k)}$--solvable are necessary and sufficient. We conclude this section giving a different proof that the global $\Gamma_{\{M_k\}}$--hypoellipticity of $L$ implies its global $\Gamma_{(M_k)}$--hypoellipticity by Komatsu levels.

Similarly to the Roumieu case, we have the same first necessary condition for the global $\Gamma_{(M_k)}$--hypo\-ellipticity of $L$:
\begin{prop}\label{nechbeu1}
	If $L$ is globally $\Gamma_{(M_k)}$--hypoelliptic then $\mathcal{N}$ is a finite set.
\end{prop}
In the next proposition we give the second necessary condition for the global $\Gamma_{(M_k)}$--hypoellipticity of $L$.

\begin{prop}\label{nechbeu2}
	If $L$ is globally $\Gamma_{(M_k)}$--hypoelliptic, then there exist $C,N>0$ such that
	\begin{equation}\tag{\ref{hypothesis5}}
	|\lambda_m(\xi)+a\mu_r(\eta)|\geq C \exp \{-M(N (\langle \xi \rangle + \langle \eta \rangle))\}, 
	\end{equation}
	for all  $[\xi] \in \widehat{G_1}, \ [\eta] \in \widehat{G_2},  1 \leq m \leq d_\xi,  1 \leq r \leq d_\eta$ whenever $\lambda_m(\xi)+a\mu_r(\eta) \neq 0$.
\end{prop}
\begin{proof}
	Let us prove the result by contradiction. Let us assume that \eqref{hypothesis5} is not satisfied, then for every $K \in \N$, we can choose a $[\xi_K] \in  \widehat{G_1}$ and a $[\eta_K] \in \widehat{G_2}$ such that
	\begin{equation}\label{condition4}
	0 < | \lambda_{\tilde{m}}(\xi_K)+ a \mu_{\tilde{r}}(\eta_K)| \leq \exp\{-M(K(\langle \xi_K \rangle + \langle \eta_K \rangle))\},
	\end{equation}
	for some $1\leq \tilde{m} \leq d_{\xi_K}$ and $1 \leq \tilde{r} \leq d_{\xi_K}$. We can assume that $\jp{\xi_j}+\jp{\eta_j} \leq \jp{\xi_\ell} + \jp{\eta_\ell}$ when $j\leq \ell$. 
	
	Let $\mathcal{A}=\{([\xi_j],[\eta_j])\}_{j \in \N}$. It is easy to see that $\mathcal{A}$ has infinitely many elements. Define 
	$$\doublehat{\, u \,}\!(\xi,\eta)_{mn_{rs}} = \left\{
	\begin{array}{ll}
	1, & \mbox{if $([\xi],[\eta])=([\xi_j],[\eta_j])$ for some $j\in \N$ and \eqref{condition4} is satisfied} , \\
	0, & \mbox{otherwise.}
	\end{array}
	\right.
	$$
	By \eqref{caracbeurling} and \eqref{distbeur}, it is easy to see that $u \in \Gamma'_{(M_k)}(G)\backslash \Gamma_{(M_k)}(G)$. Let us show  that ${Lu=f \in \Gamma_{(M_k)}(G)}$.
	
	If  $([\xi],[\eta]) \neq ([\xi_j],[\eta_j])$ for any $j\in \N$ then $\doublehat{\,f\,}\!(\xi,\eta)=0$. In the other hand, for every $K \in \N$, we have
	\begin{align}
	|\doublehat{\,f\,}\!(\xi_K,\eta_K)_{\tilde{m}1_{\tilde{r}1}}|	&=|\lambda_{\tilde{m}}(\xi_K)+a \mu_{\tilde{r}}(\eta_K)||\doublehat{\, u \,}\!(\xi_K,\eta_K)_{\tilde{m}1_{\tilde{r}1}}|\nonumber\\
	&\leq \exp\{-M(K(\langle \xi_K\rangle +\langle \eta_K \rangle))\}\nonumber
	\end{align}
	Therefore $Lu=f \in \Gamma_{(M_k)}(G)$, which contradicts the hypothesis.
\end{proof}

\begin{prop}\label{sufhbeu}
	Assume that $\mathcal{N}$ is a finite set and there exist $C,N>0$ such that
	\begin{equation}\tag{\ref{hypothesis5}}
	|\lambda_m(\xi)+a\mu_r(\eta)|\geq C \exp \{-M(N (\langle \xi \rangle + \langle \eta \rangle))\}, 
	\end{equation}
	for all  $[\xi] \in \widehat{G_1}, \ [\eta] \in \widehat{G_2},  1 \leq m \leq d_\xi,  1 \leq r \leq d_\eta$ whenever $\lambda_m(\xi)+a\mu_r(\eta) \neq 0$, then $L$ is globally $\Gamma_{(M_k)}$--hypoelliptic.
\end{prop}

\begin{proof} 
	Suppose $Lu=f \in \Gamma_{(M_k)}(G)$ for some $u \in \Gamma'_{(M_k)}(G)$.  Since $\mathcal{N}$ is finite, it is enough to study the behavior of $\widehat{u}(\xi,\eta)_{mn_{rs}}$ outside of $\mathcal{N}$. If $([\xi],[\eta]) \notin \mathcal{N}$, then
	$$
	\doublehat{\, u \,}\!(\xi,\eta)_{mn_{rs}} = -i( \lambda_m (\xi) + a \mu_r(\eta)^{-1} \doublehat{\,f\,}\!(\xi,\eta)_{mn_{rs}},
	$$
	for all $1 \leq m \leq d_\xi$ and $1 \leq r \leq d_\eta$. Thus
	\begin{align}
	|\doublehat{\, u \,}\!(\xi,\eta)_{mn_{rs}}| &= |{\lambda_m(\xi)+ a \mu_r(\eta)}|^{-1} |\doublehat{\,f\,}\!(\xi,\eta)_{mn_{rs}}| \nonumber \\
	&\leq C\exp\{M(N(\langle \xi\rangle +\langle \eta \rangle))\} |\doublehat{\,f\,}\!(\xi,\eta)_{mn_{rs}}|. \nonumber 
	\end{align}
	
	Since $f \in\Gamma_{(M_k)}(G)$, for every $N'>0$, there exists $C_{N'}>0$ such that
	$$
	\|\doublehat{\,f\,}\!(\xi,\eta)\|_{\HS} \leq C_{N'} \exp\{-M(N'(\langle \xi \rangle + \langle \eta \rangle)) \}.
	$$
	Hence
	$$
	|\doublehat{\, u \,}\!(\xi,\eta)_{mn_{rs}}| \leq C_{N'}\exp\{M(N(\langle \xi\rangle +\langle \eta \rangle))\}\exp\{-M(N'(\langle \xi \rangle + \langle \eta \rangle)) \}.
	$$
	Fix $D>0$. If $N \leq D$, then
	$$
	\exp\{M(N(\langle \xi\rangle +\langle \eta \rangle))\}\leq \exp\{M(D(\langle \xi\rangle +\langle \eta \rangle))\},
	$$
	for all  $[\xi] \in \widehat{G_1}, \ [\eta] \in \widehat{G_2}$, because $M$ is a non-decreasing function, as well the exponential. So
	$$
	|\doublehat{\, u \,}\!(\xi,\eta)_{mn_{rs}}|  \leq C_{N'}\exp\{M(D(\langle \xi\rangle +\langle \eta \rangle))\}\exp\{-M(N'(\langle \xi \rangle + \langle \eta \rangle)) \}.
	$$
	Choose $N'= DH$. By \eqref{komine} we have  
	$$
	\exp\{\textstyle- M(DH(\langle \xi \rangle + \langle \eta \rangle))\} \leq 
	A\exp\{-2M(D(\langle \xi \rangle + \langle \eta \rangle))\}.
	$$
	Thus
	\begin{align}
	|\doublehat{\, u \,}\!(\xi,\eta)_{mn_{rs}}|  &\leq C_D \exp\{M(D(\langle \xi\rangle +\langle \eta \rangle))\}\exp\{-M(DH(\langle \xi \rangle + \langle \eta \rangle)) \} \nonumber \\ &\leq C_{D}\exp\{M(D(\langle \xi\rangle +\langle \eta \rangle))\}\exp\{-2M(D(\langle \xi \rangle + \langle \eta \rangle)) \} \nonumber \\
	&\leq  C_D\exp\{-M(D(\langle \xi\rangle +\langle \eta \rangle))\}. \nonumber 
	\end{align}
	
	If $N>D$, choose $N'=NH$. Again by \eqref{komine}, 
	$$
	\exp\{\textstyle- M(NH(\langle \xi \rangle + \langle \eta \rangle))\} \leq 
	A\exp\{-2M(N(\langle \xi \rangle + \langle \eta \rangle))\}.
	$$
	So
	\begin{align}
	|\doublehat{\, u \,}\!(\xi,\eta)_{mn_{rs}}|  &\leq C\exp\{M(N(\langle \xi\rangle +\langle \eta \rangle))\}\exp\{-M(NH(\langle \xi \rangle + \langle \eta \rangle)) \}\nonumber \\
	&\leq C\exp\{M(N(\langle \xi\rangle +\langle \eta \rangle))\}\exp\{-2M(N(\langle \xi \rangle + \langle \eta \rangle)) \} \nonumber \\
	&\leq  C\exp\{-M(N(\langle \xi\rangle +\langle \eta \rangle))\} \nonumber \\
	& \leq C\exp\{-M(D(\langle \xi\rangle +\langle \eta \rangle))\}. \nonumber 
	\end{align}
	
	Hence, for every $D>0$, there exists $C_D>0$ such that
	$$
	|\doublehat{\, u \,}\!(\xi,\eta)_{mn_{rs}}|  \leq C_D\exp\{-M(D(\langle \xi \rangle + \langle \eta \rangle)),
	$$
	for all $([\xi],[\eta])\notin\mathcal{N}$. Therefore $u \in \Gamma_{(M_k)}(G)$. 
	
\end{proof} 

The next propositions are intended to help the proof of Theorem \ref{solvabeur}.
\begin{prop}\label{voltasolvbeur}
	Assume that there exist $C,N>0$ such that
	\begin{equation}\tag{\ref{hypothesis5}}
	|\lambda_m(\xi)+a\mu_r(\eta)|\geq C \exp \{-M(N (\langle \xi \rangle + \langle \eta \rangle))\}, 
	\end{equation}
	for all  $[\xi] \in \widehat{G_1}, \ [\eta] \in \widehat{G_2},  1 \leq m \leq d_\xi,  1 \leq r \leq d_\eta$ whenever $\lambda_m(\xi)+a\mu_r(\eta) \neq 0$, then $L$ is globally $\Gamma'_{(M_k)}$--solvable. 	
\end{prop}
\begin{proof} For each $f \in \mathcal{K}$ define
	\begin{equation*}\label{solutionbe}
	\doublehat{\, u \,}\!(\xi,\eta)_{mn_{rs}} = \left\{
	\begin{array}{ll}
	0, & \mbox{if } \lambda_m(\xi)+ a \mu_r(\eta)=0, \\
	-i({\lambda_m(\xi)+ a \mu_r(\eta)})^{-1} \doublehat{\,f\,}\!(\xi,\eta)_{mn_{rs}}, & \mbox{otherwise. }
	\end{array}
	\right.
	\end{equation*}
	Let us show that $\{ \doublehat{\, u \,}\!(\xi,\eta)_{mn_{rs}} \}$ is the sequence of Fourier coefficients of an element $u \in \Gamma'_{(M_k)}(G).$ We have by hypothesis that
	\begin{align}
	|\doublehat{\, u \,}\!(\xi,\eta)_{mn_{rs}}| &=  |{\lambda_m(\xi)+ a \mu_r(\eta)}|^{-1}|\doublehat{\,f\,}\!(\xi,\eta)_{mn_{rs}}|\nonumber \\
	& \leq  C \exp{\{M(N(\langle \xi \rangle + \langle \eta \rangle))\}}  |\doublehat{\,f\,}\!(\xi,\eta)_{mn_{rs}}|. \nonumber
	\end{align} 
	Using the fact that $f \in \Gamma'_{(M_k)}(G)$, we conclude that there exist $C,\, N'>0$ such that
	$$
	|\doublehat{\, u \,}\!(\xi,\eta)_{mn_{rs}}| \leq C\exp{\{M(N(\langle \xi \rangle + \langle \eta \rangle))\}}\exp{\{M(N'(\langle \xi \rangle + \langle \eta \rangle))\}},
	$$
	for all $[\xi] \in \widehat{G_1}$, $[\eta] \in \widehat{G_2}$, $1 \leq m,n\leq d_\xi$ and $1 \leq r,s \leq d_\eta$. Let $D=\max\{B,N\}$, so
	$$
	|\doublehat{\, u \,}\!(\xi,\eta)_{mn_{rs}}| \leq C\exp{\{2M(D(\langle \xi \rangle + \langle \eta \rangle))\}} \leq C\exp{\{M(\widetilde{D}(\langle \xi \rangle + \langle \eta \rangle))\}}, 
	$$
	where $\widetilde{D}=DH$. Therefore $u \in \Gamma'_{(M_k)}(G)$ and $Lu=f$.

\end{proof}	
\begin{prop}\label{idasolvbeu}
	If $L$ is globally $\Gamma'_{(M_k)}$--solvable, then there exist $C,N>0$ such that
	\begin{equation}\tag{\ref{hypothesis5}}
	|\lambda_m(\xi)+a\mu_r(\eta)|\geq C \exp \{-M(N (\langle \xi \rangle + \langle \eta \rangle))\}, 
	\end{equation}
	for all  $[\xi] \in \widehat{G_1}, \ [\eta] \in \widehat{G_2},  1 \leq m \leq d_\xi,  1 \leq r \leq d_\eta$ whenever $\lambda_m(\xi)+a\mu_r(\eta) \neq 0$. Moreover, for any admissible ultradifferentiable function $f \in  \Gamma_{(M_k)}(G)$,  there exists  $u \in \Gamma_{(M_k)}(G)$ such that $Lu=f$. 
\end{prop}
\begin{proof}
	Suppose that is not true, then for any $K\in \N$  there exist $[\xi_{K}] \in\widehat{G_1}$ and $[\eta_{K}] \in \widehat{G_2}$ satisfying
	\begin{equation}\label{condition6}
	0 <| \lambda_{\tilde{m}}(\xi_{K})+ a \mu_{\tilde{r}}(\eta_{K})| \leq \frac{1}{K}\exp\{-M(K(\langle \xi_{K} \rangle + \langle \eta_{K} \rangle))\},
	\end{equation}
	for some $1\leq \tilde{m} \leq d_{\xi_{K}}$ and $1 \leq \tilde{r} \leq d_{\xi_{K}}$.  We can assume that $\jp{\xi_j}+\jp{\eta_j} \leq \jp{\xi_\ell} + \jp{\eta_\ell}$ when $j\leq \ell$.
	
	Define 
	$$\doublehat{\, f \,}\!(\xi,\eta)_{mn_{rs}} = \left\{
	\begin{array}{ll}
	1, & \mbox{if $([\xi],[\eta])=([\xi_j],[\eta_j])$ for some $j\in \N$ and \eqref{condition6} is satisfied} , \\
	0, & \mbox{otherwise.}
	\end{array}
	\right.
	$$
	Notice that $f \in\mathcal{K}$. If $Lu=f$ for some $u \in \Gamma'_{\{M_k\}}(G)$, then
	$$
	\doublehat{\, u \,}\!(\xi_{K},\eta_{K})_{\tilde{m}1_{\tilde{r}1}} = -i(\lambda_{\tilde{m}}(\xi_{K})+ a \mu_{\tilde{r}}(\eta_{K}))^{-1} \doublehat{\,f\,}\!(\xi_{K},\eta_{K})_{\tilde{m}1_{\tilde{r}1}}.
	$$
	So
	\begin{align}
	|\doublehat{\, u \,}\!(\xi_{K},\eta_{K})_{\tilde{m}1_{\tilde{r}1}}|
	& =   |\lambda_{\tilde{m}}(\xi_{K})+ a \mu_{\tilde{r}}(\eta_{K})|^{-1}  |\doublehat{\,f\,}\!(\xi_{K},\eta_{K})_{\tilde{m}1_{\tilde{r}1}}| \nonumber \\
	& \geq K \exp\{ M(K(\langle \xi_{K} \rangle + \langle \eta_{K} \rangle))\}, \nonumber 
	\end{align}
	which implies that $u \notin \Gamma'_{(M_k)}(G)$, a contradiction. 

	The proof of the last part of the theorem is analogous to the proof of Proposition \ref{necsolrou} and so its proof is omitted.
	
	\end{proof}
 
 We can summarize the connection between the different notions of global hypoellipticity and global solvability of constant coefficients vector fields in the following diagram:
 \begin{equation}\label{diagram}
 \begin{array}{ccccc}
 GH & \Longrightarrow & G\Gamma_{\{M_k\}}H & \Longrightarrow & G\Gamma_{(M_k)}H\\
 \Big\Downarrow & & \Big\Downarrow & & \Big\Downarrow\\
 GS & \Longrightarrow & G\Gamma'_{\{M_k\}}S & \Longrightarrow & G\Gamma'_{(M_k)}S
 \end{array}
 \end{equation}

We can prove that the global $\Gamma_{\{M_k\}}(G)$--hypoellipticity of the operator $L$  implies its global $\Gamma_{(M_k)}$--hypoellipticity using what we will call Komatsu levels.
\begin{defi}
	Let $\{M_k\}_{k \in \N}$ be a sequence satisfying the conditions (M.0)--(M.3') and let $N>0$. The Komatsu Level $N$ of ultradifferentiable functions $\Gamma^N_{M_k}(G)$ is the space of $C^\infty$ functions $f$ on $G$ such that there exists $C>0$ satisfying
	$$
	\|\widehat{f}(\phi)\|_{\HS} \leq C \exp\{-M(N\jp{\phi}) \},
	$$
	for all $[\phi] \in \widehat{G}$, $1 \leq i,j \leq d_\phi$.
\end{defi}
Notice that this definition is independent of the choice of the representative of $[\phi] \in \widehat{G}$. Moreover, we have 
\begin{equation}\label{level}
\Gamma_{\{M_k\}}(G) = \bigcup_{N>0} \Gamma_{M_k}^N(G) \quad \quad \mbox{and} \quad \quad  \Gamma_{(M_k)}(G) = \bigcap_{N>0} \Gamma_{M_k}^N(G).
\end{equation}

Let us investigate how the operator $L$ acts on Komatsu levels. For $u\in \Gamma_{M_k}^N(G)$, we obtain from \eqref{fourierLu}
$$
\doublehat{\,Lu\,}\!(\xi,\eta)_{mn_{rs}}=i(\lambda_m(\xi)+a \mu_r(\eta))\doublehat{\, u\,}\!(\xi,\eta)_{mn_{rs}}.
$$
By \eqref{symbol}, we have $|\lambda_m(\xi)| \leq \jp{\xi}$ and $|\mu_r(\eta)| \leq \jp{\eta}$, so we have
$$
\|\doublehat{\,Lu\,}\!(\xi,\eta)\|_{\HS} \leq C (\jp{\xi}+\jp{\eta})\|\doublehat{\,u\,}\!(\xi,\eta)\|_{\HS}.
$$
By \eqref{propM1}, there exists $C>0$ such that $\jp{\xi}+\jp{\eta} \leq C\exp\{\tfrac{1}{2}M(N(\jp{\xi}+\jp{\eta})) \}$. Using now \eqref{komine}, we obtain
$$
\|\doublehat{\,Lu\,}\!(\xi,\eta)\|_{\HS} \leq C\exp\left\{-M\left(\tilde{N}(\jp{\xi}+\jp{\eta})\right) \right\},
$$
where $\tilde{N}=\frac{N}{H}$, which implies that $Lu \in \Gamma_{M_k}^{\tilde{N}}(G)$.

Assume that $L$ is globally $\Gamma_{\{M_k\}}$--hypoelliptic. In the proof of Proposition \ref{voltahyporou} we showed that if $Lu \in \Gamma^N_{M_k}(G)$, then $u \in \Gamma^{\tilde{N}}_{M_k}(G)$, where $\tilde{N} = \frac{N}{H}$. 
Let us prove that $L$ is globally $\Gamma_{(M_k)}(G)$--hypoelliptic. If $Lu \in \Gamma_{(M_k)}(G)$, by  \eqref{level} we get $Lu \in \Gamma_{M_k}^N(G)$, for all $N>0$ and then $u\in \Gamma_{M_k}^{\tilde{N}}(G)$, for all $\tilde{N}>0$. Therefore $u\in \Gamma_{(M_k)}(G)$ and $L$ is globally $\Gamma_{(M_k)}$--hypoelliptic.

We also can prove that global $\Gamma_{\{M_k\}}$--solvability implies global $\Gamma_{(M_k)}$--solvability for the operator $L$ using Komatsu levels of ultradistributions.
\begin{defi}
	Let $\{M_k\}_{k \in \N}$ be a sequence satisfying the conditions (M.0), (M.1), (M.2) and (M.3') and $N>0$. The  Komatsu Level $N$ of ultradistributions $\Gamma^{'N}_{M_k}(G)$ is the space of linear functionals $u$ such that there exists $C>0$ satisfying
	$$
	\|\widehat{u}(\phi)\|_{\HS} \leq C \exp\{M(N\jp{\phi}) \},
	$$
	for all $[\phi] \in \widehat{G}$, $1 \leq i,j \leq d_\phi$.
\end{defi}
Similarly, we have
\begin{equation}\label{level2}
\Gamma'_{\{M_k\}}(G) = \bigcap_{N>0} \Gamma_{M_k}^{'N}(G) \quad \quad \mbox{and} \quad \quad  \Gamma'_{(M_k)}(G) = \bigcup_{N>0} \Gamma_{M_k}^{'N}(G).
\end{equation}

Suppose that $L$ is globally $\Gamma_{\{M_k\}}'$--solvable. In the proof of Proposition \ref{voltasolvrou} we showed that if $f$ is an admissible ultradistribution and $f\in\Gamma^{'N}_{M_k}(G)$, then there exists $u\in \Gamma^{'\tilde{N}}_{M_k}(G)$ such that $Lu=f$, where $\tilde{N}=NH$. 

Let us prove that $L$ is globally $\Gamma_{(M_k)}'$--solvable. Let $f\in \Gamma'_{(M_k)}(G)$ an admissible ultradistribution. By \eqref{level2}, $f\in\Gamma_{M_k}^{'N}(G)$ for some $N>0$ and then there exists $u\in \Gamma^{'\tilde{N}}_{M_k}(G)$ such that $Lu=f$, where $\tilde{N}=NH$. Therefore $L$ is globally $\Gamma_{(M_k)}'$--solvable.

\section{Low-order perturbations}\label{low-order-pert}

We can characterize the global hypoellipticity and global solvability of the operator
$$
L=X+q,
$$
where $X\in \mathfrak{g}$ and $q\in \C$, on Komatsu classes, both Roumieu and Beurling type, similarly to the vector field case. We say that $L_q$ is globally $\Gamma_{\{M_k\}}(G)$--solvable if $L_q(\Gamma_{\{M_k\}}(G)) = \mathcal{K}_q$, where $$\mathcal{K}_q:=\{w \in \Gamma_{\{M_k\}}(G); \ {\widehat{w}}(\xi)_{mn}=0, \textrm{whenever } \lambda_m(\xi)-iq=0 \}.$$
Analogously we define the global $\Gamma_{(M_k)}(G)$--solvability of $L_q$.
\begin{thm}\label{pert}
	The operator $L_q=X+q$ is globally $\Gamma_{(M_k)}$-hypoelliptic (respectively, globally $\Gamma_{\{M_k\}}$-hypoelliptic) if and only if the following conditions hold:
	\begin{enumerate}[1.]
		\item The set
		$$
		\mathcal{N}=\{[\xi]\in \widehat{G}; \lambda_m(\xi)-iq = 0, \mbox{ for some } 1 \leq m \leq d_\xi \}
		$$
		is finite.
		\item $\exists N>0$ (respectively, $\forall N>0$) and $\exists C>0$ such that
		$$
		|\lambda_m(\xi)-iq| \geq C \exp\{-M(N\jp{\xi}) \},
		$$
		for all $[\xi] \in \widehat{G}$, $1\leq m \leq d_\xi$, whenever $\lambda_m(\xi)-iq\neq0$.
	\end{enumerate}
	Moreover, the operator $L_q$ is globally $\Gamma_{(M_k)}$-solvable (respectively, globally $\Gamma_{\{M_k\}}$-solvable) if and only if Condition 2 above is satisfied.
\end{thm}
The proof is similar to the vector field case and it will be omitted. We also have \eqref{diagram} for this case.

The next step for the study of low order perturbations is to consider the operator $L_q:= X+q$,  where $q \in \Gamma_{\{M_k\}}(G)$. The idea is to establish a connection between the global hypoellipticity and the global solvability in Komatsu sense of $L_q$ and $L_{q_0} = X+q_0$, where $q_0$ is the average of $q$ in $G$.

 In \cite{Ber94}, Bergamasco proved that the operator 
 $$
 L_q = \partial_t+ a\partial_x+q,
 $$
 where $a \in \mathbb{R}$ is an irrational non-Liouville number and $q \in C^\infty(\mathbb{T}^2)$, is globally hypoelliptic if and only if the operator  $L_{q_{0}}=\partial_t+a\partial_x+q_{0}$ is globally hypoelliptic, where $q_{0}=\int_{\mathbb{T}^2} q(t,x)\, dxdt$. The key to make this connection is the fact that $L_q \circ e^{-Q} = e^{-Q} \circ L_{q_0}$, where $Q\in C^\infty(\mathbb{T}^2)$ satisfies $(\partial_t+a\partial_x)Q=q-q_0$. The existence of such $Q$ is guaranteed by the global hypoellipticity of the operator $\partial_t+a\partial_x$.
 
For the study of the operator $L=X+q$, with $q\in \Gamma_{\{M_k\}}(G)$, we can not assume the global hypoellipticity of $X$ in view of the Greenfield-Wallach's conjecture. Hence, we will assume as hypothesis that there exists $Q \in \Gamma_{\{M_k\}}(G)$ such that
$$
XQ=q-q_0,
$$
where $q_0=\int_{G}q(x) \, dx$. 

From Proposition 5.8 of \cite{KMR19b}, we have that  
\begin{equation}\label{normallow}
L_q \circ e^{-Q} = e^{-Q} \circ L_{q_0}.
\end{equation}

Henceforth, we will assume that the sequence $\{M_k\}_{_{k\in\N_0}}$ satisfies the following additional condition:
\begin{description}
	\item[(M.4)] $\displaystyle\frac{M_r}{r!} \frac{M_s}{s!} \leq \frac{M_{r+s}}{(r+s)!}, \quad \forall r,s \in \N_0.$
\end{description}
\begin{lemma}\label{autoexp}
If $f \in \Gamma_{\{M_k\}}(G)$, then $e^{f} \in \Gamma_{\{M_k\}}(G)$.
\end{lemma}
\begin{proof}
	By the characterization of ultradifferentiable function of Roumieu type, there exist $C,h>0$ such that
	$$
	|\partial^\alpha f(x)| \leq Ch^{|\alpha|}M_{|\alpha|},
	$$
	for all $\alpha \in \N_0^d$, $x\in G$.
	
	Let $\alpha \in \N_0^d$ such that $|\alpha|=p$. 
	We have that
	$$
	|\partial^\alpha e^{f(x)}| \leq |e^{f(x)}|h^p \sum_{k=1}^p C^k\left( \sum_{\lambda \in \Delta(p,k)}\binom{p}{\lambda} \frac{1}{r(\lambda)!}\prod_{j=1}^k M_{\lambda_j} \right),
	$$
	where $\Delta(p,k) = \{\lambda\in\N^k; |\lambda|=p \mbox{ and } \lambda_1\geq \cdots \geq \lambda_k \geq 1 \}$ and $r(\lambda) \in \N_0^d$, where $r(\lambda)_j$ counts how many times $j$ appears on $\lambda$. For example, $\lambda=(2,2,1,1) \in \Delta(6,4)$ and $r(\lambda)=(2,2,0,0,0,0)$. Since $\binom{p}{\lambda} = \frac{p!}{\lambda_1! \cdots \lambda_k!}
	$, by property (M.4) we obtain
	$$
	\binom{p}{\lambda}\prod_{j=1}^k M_{\lambda_j}  = p! \prod_{j=1}^k \frac{M_{\lambda_j}}{\lambda_j!} \leq p! \frac{M_{|\lambda|}}{|\lambda|!} = M_p.
	$$
Then
	$$
		|\partial^\alpha e^{f(x)}| \leq K h^p M_p \sum_{k=1}^pC^k\sum_{\lambda \in \Delta(p,k)}\frac{1}{r(\lambda)!}.
$$

	We have that 
	$$
	\sum_{k=1}^pC^k\sum_{\lambda \in \Delta(p,k)}\frac{1}{r(\lambda)!} = \sum_{k=1}^p \binom{p-1}{k-1}\frac{C^k}{k!}.
	$$ 

Therefore,
$	\sum\limits_{k=1}^pC^k\sum\limits_{\lambda \in \Delta(p,k)}\frac{1}{r(\lambda)!} \leq 2^p e^C$
and we obtain
$$
|\partial^\alpha e^{f(x)}| \leq K (2h)^p M_p,
$$
which implies that $e^f \in \Gamma_{\{M_k\}}(G)$.
\end{proof}
\begin{rem}
With a slight modification in the above proof it is possible to show that $e^f \in \Gamma_{(M_k)}(G)$ whenever $f\in \Gamma_{(M_k)}(G)$.
\end{rem}
From Lemma \ref{autoexp}, we obtain that $e^Q v \in \Gamma_{\{M_k\}}(G)$, whenever $v\in \Gamma_{\{M_k\}}(G)$. Moreover, for $u \in \Gamma'_{\{M_k\}}(G)$, we also have $e^Qu \in \Gamma'_{\{M_k\}}(G)$. The equality \eqref{normallow} motivates us to define the global $\Gamma'_{\{M_k\}}$--solvability of $L_q$ as:

\begin{defi}\label{definitionsolv}
	Let $G$ be a compact Lie group, $X\in \mathfrak{g}$, and $Q\in \Gamma_{\{M_k\}}(G)$. We say that the operator $$L_q=X+q,$$ where $XQ=q-q_0$, $q_0=\int_G q(x)\, dx$,  is globally $\Gamma'_{\{M_k\}}$--solvable if  $L_{q}(\DG) = \mathcal{J}_{q}$, where
		$$
		\mathcal{J}_q := \{v \in \Gamma'_{\{M_k\}}(G); \ e^Q v \in \mathcal{K}_{q_0} \}.
		$$
\end{defi}

\begin{prop}\label{connection}
	Let $G$ be a compact Lie group and consider the operator
	$
	L=X+q,
	$
	where $X\in \mathfrak{g}$ and $q\in \Gamma_{\{M_k\}}(G)$. Assume that there exists $Q \in \Gamma_{\{M_k\}}(G)$ satisfying $XQ=q-q_0,
	$
	where $q_0=\int_{G}q(x) \, dx$. The operator $L_{q}$ is globally $\Gamma_{\{M_k\}}$--hypoelliptic  if and only if $L_{q_0}$ is globally $\Gamma_{\{M_k\}}$--hypoelliptic. Moreover, the operator $L_{q}$ is globally $\Gamma_{\{M_k\}}$--solvable if and only if $L_{q_0}$ is globally $\Gamma_{\{M_k\}}$--solvable.
\end{prop}
\begin{proof}
	Suppose that $L_q$ is globally $\Gamma_{\{M_k\}}(G)$--hypoelliptic. If $L_{q_0}u = f \in \Gamma_{\{M_k\}}(G)$ for some $u\in \Gamma'_{\{M_k\}}(G)$, then
	$e^{-Q}L_{q_0}u = e^{-Q} f \in \Gamma_{\{M_k\}}(G)$. Since $ e^{-Q} \circ L_{q_0}=L_{q} \circ e^{-Q},$ we have
	$L_q(e^{-Q} u) \in \Gamma_{\{M_k\}}(G)$ and by global $\Gamma_{\{M_k\}}(G)$--hypoellipticity of $L_{q}$ we have $e^{-Q}u \in \Gamma_{\{M_k\}}(G)$, which implies that $u\in \Gamma_{\{M_k\}}(G)$ and then $L_{q_0}$ is globally $\Gamma_{\{M_k\}}(G)$--hypoelliptic.
	
	Assume now that $L_{q_0}$ is globally $\Gamma_{\{M_k\}}(G)$--hypoelliptic. If $L_q u = f \in \Gamma_{\{M_k\}}(G)$ for some $u\in \Gamma'_{\{M_k\}}(G)$, we can write $L_q (e^{-Q} e^Q u) = f \in \Gamma_{\{M_k\}}(G)$. By the fact that $L_{q} \circ e^{-Q} = e^{-Q} \circ L_{q_0} $ we obtain $e^{-Q} L_{q_0}(e^Q u) = f$, that is, $L_{q_0}(e^Q u) = e^Q f \in \Gamma_{\{M_k\}}(G)$. By global $\Gamma_{\{M_k\}}(G)$--hypoellipticity of $L_{q_0}$ we have that $e^Q u \in \Gamma_{\{M_k\}}(G)$ and then $u \in \Gamma_{\{M_k\}}(G)$.

	Assume that $L_q$ is globally solvable and let $f \in \mathcal{K}_{q_0}$. Let us show that there exists $u\in \Gamma'_{\{M_k\}}(G)$ such that $L_{q_0}u=f$. We can write $f=e^{Q}e^{-Q}f$, so $e^{-Q }f \in \mathcal{J}_{q}$. Since $L_q$ is globally solvable, there exists $v \in \Gamma'_{\{M_k\}}(G)$ such that $L_q v = e^{-Q }f $. we can write $v=e^{-Q} e^Q v$ and then $L_q (e^{-Q} e^Q v) = e^{-Q }f $. By \eqref{normallow}, we have
	$$
	e^{-Q}L_{q_0}e^Q v = L_q (e^{-Q} e^Q v) = e^{-Q }f,
	$$
	that is, $L_{q_0}e^Q v=f$.
	
	Suppose now that $L_{q_0}$ is globally solvable and let $f \in \mathcal{J}_q$. By the definition of $\mathcal{J}_q$, we have $e^Q f \in \mathcal{K}_{q_0}$ and by the global solvability of $L_{q_0}$, there exists $u \in \Gamma'_{\{M_k\}}(G)$ such that $L_{q_0} u = e^Q f$, that is, $e^{-Q}L_{q_0}u = f$. By \eqref{normallow}, we obtain $L_q e^{-Q}u =f$.
	
\end{proof}
\begin{cor}
	If $L_q$ is globally $\Gamma_{\{M_k\}}$--hypoelliptic, then $L$ is globally $\Gamma_{\{M_k\}}$--solvable.
\end{cor}

\begin{ex}
	Let $G=\mathbb{T}^1\times \St$ and $\{M_k \}_{k \in \N_0}$ the sequence given by $M_k=(k!)^s$, with $s \geq 1$. So, the Komatsu class of Roumieu type associated to this sequence is the Gevrey space $\gamma^s(\mathbb{T}^1\times \St)$. Consider the continued fraction  $\alpha=\left[10^{1 !}, 10^{2 !}, 10^{3 !}, \ldots\right]$ and the vector field $X \in \mathfrak{s}^3$. Using rotation on $\St$, without loss of generality, we may assume that $X$ has the symbol 
	$$
	\sigma_{X}(\ell)_{mn}=im\delta_{mn}, \quad \ell \in \tfrac{1}{2}\N_0, \ -\ell\leq m,n\leq \ell, \ \ell-m, \ell-n \in \N_0,
	$$ 
	with $\delta_{mn}$ standing for the Kronecker's delta. The details about the Fourier analysis on $\St$ can be found in Chapter 11 of \cite{livropseudo}. Consider
	$$
	L_q=\partial_t + \alpha X + q(t,x),
	$$
	where $q(t,x)=\cos(t)+h(x)+\tfrac{1}{2}i$, where $h$ is expressed in Euler's angles by
		\begin{equation*}
	h(x(\phi,\theta,\psi)) = -\cos\left(\tfrac{\theta}{2}\right)\sin\left(\tfrac{\phi+\psi}{2}\right).
	\end{equation*}
Notice that $q$ is an analytic function, which implies that $q \in \gamma^s(\mathbb{T}^1\times \St)$ for all $s \geq 1$. Let $Q(t,x)=\sin(t)+\frac{1}{\alpha}\emph{\mbox{tr}}(x)$, where $\emph{\mbox{tr}}$ is the trace function given in Euler's angles by 
	$$
\emph{\mbox{tr}}(x(\phi,\theta,\psi)) = 2\cos\left(\tfrac{\theta}{2}\right)\cos\left(\tfrac{\phi+\psi}{2}\right).
$$
The vector field $X$ is the operator $\partial_\psi$ in Euler's angle and we obtain $X\emph{\mbox{tr}}(x)=h(x)$. Hence,
$$
(\partial_t+\alpha X) Q(t,x)=q(t,x)-\tfrac{1}{2}i,
$$
and by Proposition \ref{connection} the operator $L_q$ is globally $\gamma^s$--hypoelliptic if and only if the operator
$$
L_{q_0}= \partial_t + \alpha X +\tfrac{1}{2}i
$$
is globally $\gamma^2$--hypoelliptic. For the operator $L_{q_0}$, the set
$$
\mathcal{N} = \left\{(k,\ell) \in \Z \times \tfrac{1}{2}\N_0; \   k+\alpha m+\tfrac{1}{2}=0,   \mbox{ for some }  -\ell \leq m \leq \ell, \ \ell - m \in \N \right\}
$$
is empty, because $\alpha$ is an irrational number. So the condition 1 from Theorem \ref{pert} is satisfied. The condition 2 can be written as: for any $N>0$, there exists $C_N>0$ such that
\begin{equation}\label{liouexp}
\left|k+\alpha m+\tfrac{1}{2}\right| \geq C_N \exp\{-N(k+\ell+1)^{1/s}\},
\end{equation}
for all $k \in \Z$, $\ell \in \tfrac{1}{2}\N_0$, and $-\ell \leq m \leq \ell$. By Proposition 6.2 of \cite{AKM19}, we have that $\alpha$ is not an exponential Liouville number of order $s$, for any $s \geq 1$, which implies \eqref{liouexp}. Therefore $L_{q_0}$ is globally $\gamma^s$--hypoelliptic and by Proposition \ref{connection}, we conclude that $L_q$ is globally $\gamma^s$--hypoelliptic, which implies that $L_q$ is also globally $\gamma^s$--solvable.

Moreover, we have that $\alpha$ is a Liouville number, so $L_{q_0}$ is neither globally hypoelliptic nor globally solvable in the smooth sense (see Example 3.10 of \cite{KMR19b}).  By Propositions 5.8 and 5.10 of \cite{KMR19b}, we conclude that the operator $L_q$ is neither globally hypoelliptic nor globally solvable in the smooth sense. This example illustrates why we do not have the horizontal converse arrows in \eqref{diagram}. 

Consider now the operator
$$
L_{q_1}=\partial_t+\alpha X+q_1(t,x),
$$
where $q_1(t,x)=\cos(t)+h(x)+\alpha i$. Analogously to the previous example, we have
$$
(\partial_t+\alpha X) Q(t,x)=q(t,x)-\alpha i,
$$
and by Proposition \ref{connection} the operator $L_{q_1}$ is globally $\gamma^s$--hypoelliptic if and only if the operator
$$
L_{{q_1}_{0}}= \partial_t + \alpha X +\alpha i
$$
is globally $\gamma^s$--hypoelliptic. For this operator, the set
$$
\mathcal{N} = \left\{(k,\ell) \in \Z \times \tfrac{1}{2}\N_0; \   k+\alpha m+\alpha=0,   \mbox{ for some }  -\ell \leq m \leq \ell, \ \ell - m \in \N \right\}
$$
has infinitely many elements, so the operator $L_{{q_1}_{0}}$ is not globally $\gamma^s$--hypoelliptic. However, since $\alpha$ is not an exponential Liouville number of order $s$, for any $s$, we have that for all $N>0$, there exists $C_N>0$ such that
\begin{equation*}\label{liouexp2}
|k+\alpha(m+1)| \geq C_N \exp\{-N(|k|+\ell+1)^{1/s}\},
\end{equation*}
for all $k \in \Z$, $\ell \in \tfrac{1}{2}\N_0$, and $-\ell \leq m \leq \ell$. By Theorem \ref{pert} the operator $L_{{q_1}_{0}}$ is globally $\gamma^s$--solvable, which implies that $L_{q_1}$ is globally $\gamma^s$--solvable. Again, the operator $L_{q_1}$ is neither globally hypoelliptic nor globally solvable in the smooth sense.
\end{ex}

\section*{Acknowledgments}
This study was financed in part by the Coordena\c c\~ao de Aperfei\c coamento de Pessoal de N\' ivel Superior - Brasil (CAPES) - Finance Code 001. The last
 author was also supported by the FWO Odysseus grant, by the Leverhulme Grant RPG-2017-151, and by EPSRC Grant EP/R003025/1.
\bibliographystyle{abbrv}
\bibliography{biblio}

\end{document}